\def\arxivformat{1}

\ifdefined\arxivformat
\documentclass[ECP]{ejpecparxiv} 
\else
\documentclass[ECP]{ejpecp} 
\fi

\usepackage[numbers]{natbib} %
\usepackage{hyperref} 
\usepackage{amsmath,amsbsy,amssymb,amsfonts,amsthm}
\usepackage{color}
\usepackage{xspace} %
\usepackage{breakurl}

\newcommand{\dt}{\,dt}

\newcommand{\smiddle}{\mathrel{}\middle|\mathrel{}} %

\newcommand{\eps}{\epsilon}

\def\balign#1\ealign{\begin{align}#1\end{align}}
\def\baligns#1\ealigns{\begin{align*}#1\end{align*}}
\def\balignat#1\ealign{\begin{alignat}#1\end{alignat}}
\def\balignats#1\ealigns{\begin{alignat*}#1\end{alignat*}}
\def\bitemize#1\eitemize{\begin{itemize}#1\end{itemize}}
\def\benumerate#1\eenumerate{\begin{enumerate}#1\end{enumerate}}

\newcommand{\qtext}[1]{\quad\text{#1}\quad} 

\let\originalleft\left
\let\originalright\right
\renewcommand{\left}{\mathopen{}\mathclose\bgroup\originalleft}
\renewcommand{\right}{\aftergroup\egroup\originalright}

\def\Ito{It\^o\xspace}

\def\tinycitep*#1{{\tiny\citep*{#1}}}
\def\tinycitealt*#1{{\tiny\citealt*{#1}}}
\def\tinycite*#1{{\tiny\cite*{#1}}}
\def\smallcitep*#1{{\scriptsize\citep*{#1}}}
\def\smallcitealt*#1{{\scriptsize\citealt*{#1}}}
\def\smallcite*#1{{\scriptsize\cite*{#1}}}

\def\mbb#1{\mathbb{#1}}

\def\textprod{{\textstyle\prod}} %
\def\reals{\mathbb{R}} %
\def\R{\mathbb{R}}
\def\<{\left\langle} %
\def\>{\right\rangle}

\def\defeq{\triangleq} %

\def\texthalf{{\textstyle\frac{1}{2}}}
\newcommand{\textfrac}[2]{{\textstyle\frac{#1}{#2}}} 

\def\norm#1{\left\|{#1}\right\|} %
\newcommand{\twonorm}[1]{\norm{#1}_2} %
\newcommand{\opnorm}[1]{\norm{#1}_{op}} %
\newcommand{\inner}[2]{\langle{#1},{#2}\rangle} %
\newcommand{\binner}[2]{\left\langle{#1},{#2}\right\rangle} %

\def\maxarg#1{\max\left({#1}\right)} %
\def\E{\mbb{E}} %
\def\Earg#1{\E\left[{#1}\right]}
\def\Esubarg#1#2{\E_{#1}\left[{#2}\right]}
\renewcommand{\exp}[1]{\operatorname{exp}\left(#1\right)} %
\newcommand{\grad}{\nabla} %
\newcommand{\Hess}{\nabla^2} %
\newcommand{\deriv}[2]{\frac{d #1}{d #2}} %

\newcommand{\eqnref}[1]{\eqref{eqn:#1}}

\newcommand{\lemref}[1]{Lemma~{\ref{lem:#1}}}

\newcommand{\secref}[1]{Section~{\ref{sec:#1}}}

\newcommand{\thmref}[1]{Theorem~{\ref{thm:#1}}}
\newcommand{\hset}[0]{\mathcal{H}} %

\newcommand{\wassset}{\mathcal{W}_{\norm{\cdot}}} %
\newcommand{\blset}{BL_{\norm{\cdot}}} %
\newcommand{\smoothset}{\mathcal{M}_{\norm{\cdot}}} %

\newcommand{\generator}[1]{\mathcal{A}{#1}} %
\newcommand{\genarg}[2]{(\generator{#1})({#2})} %

\newcommand{\ipm}{d_\hset} %
\newcommand{\wass}{d_{\wassset}} %
\newcommand{\bl}{d_{\blset}} %
\newcommand{\smooth}{d_{\smoothset}} %

\newcommand{\trans}[2]{P_{#1}{#2}} %
\newcommand{\transarg}[3]{(P_{#1}{#2})({#3})} %

\newcommand{\qvar}[0]{x} %
\newcommand{\pvar}[0]{z}%
\newcommand{\QVAR}[0]{\MakeUppercase{\qvar}} %
\newcommand{\PVAR}[0]{\MakeUppercase{\pvar}} %

\newcommand{\process}[2]{\PVAR_{#1,#2}} %
\renewcommand{\smoothset}{\mathcal{M}} %
\renewcommand{\wassset}{\mathcal{W}} %
\renewcommand{\blset}{\mathrm{BL}} %

\newcommand\abstractbody{We establish uniform bounds on the low-order derivatives of Stein equation solutions for a broad class of multivariate, strongly log-concave target distributions.
These ``Stein factor'' bounds deliver control over Wasserstein and related smooth function distances 
and are well-suited to analyzing the computable Stein discrepancy measures of Gorham and Mackey.
Our arguments of proof are probabilistic and feature the synchronous coupling of multiple overdamped Langevin diffusions.}

\newcommand\titlebody{Multivariate Stein Factors for a Class of Strongly Log-concave Distributions}
\SHORTTITLE{\titlebody}

\TITLE{\titlebody}

\AUTHORS{%
  Lester~Mackey\footnote{Department of Statistics, Stanford University.}
  \and %
  Jackson~Gorham\footnotemark[1]}%


\AMSSUBJ{
	60J60; %
	62E17; %
    60E15%
				  } %

\SUBMITTED{April 4, 2016} %
\ACCEPTED{August 8, 2016} %

\VOLUME{0}
\YEAR{2012}
\PAPERNUM{0}
\DOI{vVOL-PID}

\ABSTRACT{\abstractbody}

\begin{document}

\section{Introduction}
In 1972, Stein~\cite{Stein72} introduced a powerful method for bounding 
the maximum expected discrepancy, $\ipm(Q,P) \defeq \sup_{h \in \hset} |\Esubarg{Q}{h(\QVAR)} - \Esubarg{P}{h(\PVAR)} |$,
between a target distribution $P$ and an approximating distribution $Q$.
Stein's method classically proceeds in three steps:
\benumerate
\item First, one identifies a linear operator $\generator{}$ that generates mean-zero functions under the target distribution.
A common choice for a continuous target on $\reals^d$ is the infinitesimal generator of the overdamped Langevin diffusion%
\footnote{The modifier ``overdamped'' is derived from the physical analogy of an oscillator damped by friction.
}
(also known as the Smoluchowski dynamics)
\citep[Secs. 6.5 and 4.5]{Pavliotis14} with stationary distribution $P$:
\balign\label{eqn:langevin-generator}
\textstyle\genarg{u}{x} = \frac{1}{2}\inner{\grad u(x)}{\grad\log p(x)} + \frac{1}{2}\inner{\grad}{\grad u(x)}.
\ealign
Here, $p$ represents the density of $P$ with respect to Lebesgue measure.
\item Next, one shows that, for every test function $h$ in a convergence-determining class $\hset$, the \emph{Stein equation}
\balign\label{eqn:stein-equation-generator}
h(x) - \Esubarg{P}{h(\PVAR)} 
	= \genarg{u_h}{x}
\ealign
admits a solution $u_h$ in a set $\mathcal{U}$ of functions with uniformly bounded low-order derivatives.
These uniform derivative bounds are commonly termed \emph{Stein factors}.
\item Finally, one uses whatever tools necessary to upper bound the \emph{Stein discrepancy}%
\footnote{Not to be confused with the ``Stein discrepancy'' of \cite{LedouxNoPe15}, which names an entirely different quantity.}
\balign\label{eqn:stein-discrepancy}
\sup_{u \in \mathcal{U}} |\Esubarg{Q}{\genarg{u}{\QVAR}}| 
	= \sup_{u \in \mathcal{U}} |\Esubarg{Q}{\genarg{u}{\QVAR}} - \Esubarg{P}{\genarg{u}{\PVAR}}|,
\ealign
which by construction upper bounds the reference metric $\ipm(Q,P)$.
\eenumerate

To date, this recipe has been successfully used with the Langevin operator \eqnref{langevin-generator} to obtain explicit approximation error bounds for a wide variety of univariate targets $P$~\citep[see, e.g.,][]{ChenGoSh11,ChatterjeeSh11}.\footnote{In the univariate setting, the operator \eqnref{langevin-generator} is commonly called \emph{Stein's density operator}.}
The same operator has been used to analyze multivariate Gaussian approximation~\citep{Barbour90,Gotze91,ReinertRo09,ChatterjeeMe08,Meckes09,NourdinPeRe10}, but few other multivariate distributions have established Stein factors.
To extend the reach of the multivariate literature, we derive uniform Stein factor bounds for a broad class of strongly log-concave target distributions in \thmref{concave-stein-factors}.
The result covers common Bayesian target distributions, including Bayesian logistic regression posteriors under Gaussian priors, and explicitly relates the Stein discrepancy \eqnref{stein-discrepancy} and practical Monte Carlo diagnostics based thereupon~\citep{GorhamMa15} to standard probability metrics, like the Wasserstein distance.

\textbf{Notation and terminology}\quad 
We let $C^k(\reals^d)$ denote the set of real-valued functions on $\reals^d$ with $k$ continuous derivatives.
We further let $\twonorm{\cdot}$ denote the $\ell_2$ norm on $\reals^d$ 
and define the operator norms $\opnorm{v} \defeq \twonorm{v}$ for vectors $v\in \reals^d$, $\opnorm{W} \defeq \sup_{v\in \reals^d: \twonorm{v}=1}{\twonorm{Wv}}$ for matrices $W\in \reals^{d \times d}$ ,
and $\opnorm{T} \defeq \sup_{v\in \reals^d: \twonorm{v}=1}{\opnorm{T[v]}}$ for tensors $T \in \reals^{d\times d\times d}$.
We say a function $f \in C^2(\reals^d)$ is \emph{$k$-strongly concave} for $k > 0$ if 
\[v^\top\Hess f(x) v \leq -k\twonorm{v}^2,\qtext{for all} x, v \in\reals^d,\] 
and we term a function \emph{$k$-strongly log-concave} if $\log f$ is $k$-strongly concave.
We finally let $\grad^0h \defeq h$ for all functions $h$ and define the Lipschitz constants
\baligns
M_k(h) &\defeq \sup_{x,y\in\reals^d, x\neq y} \frac{\opnorm{\grad^{k-1} h(x) - \grad^{k-1} h(y)}}{\twonorm{x-y}}\ \ \text{for all}\ h\in C^{k-1}(\reals^d) \text{ and integers } k \geq 1. 
\ealigns

\section{Stein factors for strongly log-concave distributions}
Consider a target distribution $P$ on $\reals^d$ with strongly log-concave density $p$.
The following result bounds the derivatives of Stein equation solutions  
in terms of the smoothness of $\log p$ and the underlying test function $h$.
The proof, found in \secref{concave-stein-factors-proof}, is probabilistic, in the spirit of the generator method of Barbour~\cite{Barbour88} and Gotze~\cite{Gotze91}, and features the synchronous coupling of multiple overdamped Langevin diffusions.
\begin{theorem}[Stein factors for strongly log-concave distributions] \label{thm:concave-stein-factors}
Suppose that $\log p \in C^4(\reals^d)$ is $k$-strongly concave with $M_3(\log p) \leq L_3$ and $M_4(\log p) \leq L_4$.
For each $x \in\reals^d$, let $(\process{t}{x})_{t\geq0}$ represent the overdamped Langevin diffusion with infinitestimal generator
\eqnref{langevin-generator} and initial state $\process{0}{x} = x$.
Then, for each Lipschitz $h \in C^3(\reals^d)$, the function
\[
	u_h(x) \defeq \int_0^\infty \Esubarg{P}{h(\PVAR)} - \Earg{h(\process{t}{x})} \ dt
\]
solves the the Stein equation
\eqnref{stein-equation-generator}
and satisfies
\baligns
	 M_1(u_h) %
		\leq\ &\frac{2}{k}  M_1(h), \quad
	 M_2(u_h) %
		\leq \frac{2L_3}{k^2} M_1(h)+\frac{1}{k} M_2(h),\ and \\
	 M_3(u_h) %
		\leq\ &\left(\frac{6L_3^2}{k^3} +\frac{L_4}{k^2}\right) M_1(h)
		+ \frac{3L_3}{k^2} M_2(h) + \frac{2}{3k}M_3(h).
\ealigns
\end{theorem}

\thmref{concave-stein-factors} implies that the Stein discrepancy \eqnref{stein-discrepancy} with set
\[
	\mathcal{U} 
		\defeq \left\{ u \in C^2(\reals^d)  \bigg|
		\maxarg{\frac{M_1(u)}{\frac{2}{k}},\frac{M_2(u)}{\frac{2L_3}{k^2}+\frac{1}{k}}, \frac{M_3(u)}{(\frac{6L_3^2}{k^3}+\frac{L_4+3L_3}{k^2}+\frac{2}{3k})}} \leq 1		\right\}
\]
bounds the \emph{smooth function distance}
$\smooth(Q,P) = \sup_{h \in \smoothset} |\Esubarg{Q}{h(\QVAR)} - \Esubarg{P}{h(\PVAR)} |$ for
\[
	\textstyle\smoothset 
		\defeq \left\{h \in C^3(\reals^d) \smiddle
			\maxarg{M_1(h), M_2(h), M_3(h)} \leq 1 \right\}.
\]
Our next result shows that control over the smooth function distance also grants control over the \emph{$L_1$-Wasserstein distance}
(also known as the Kantorovich-Rubenstein or earth mover's distance), $\wass(Q,P) = \sup_{h \in \wassset} |\Esubarg{Q}{h(\QVAR)} - \Esubarg{P}{h(\PVAR)} |$, 
and the \emph{bounded-Lipschitz metric}, $\bl(Q,P) = \sup_{h \in \blset} |\Esubarg{Q}{h(\QVAR)} - \Esubarg{P}{h(\PVAR)} |$,
which exactly metrizes convergence in distribution on $\reals^d$. These metrics govern the test function classes
\[
\textstyle\wassset\defeq\{h : \reals^d\to\reals \mid M_1(h) \leq 1 \} 
\qtext{and}
\textstyle\blset\defeq \wassset\cap\{h : \reals^d\to\reals \mid \sup_{x\in\reals^d} |h(x)| \leq 1 \}.
\]
\begin{lemma}[Smooth-Wasserstein inequality]\label{lem:smooth-wass}
If $\mu$ and $\nu$ are probability measures on $\reals^d$ with finite means, and $G \in\reals^d$ is a standard normal random vector, then
\baligns
\max(\bl(\mu,\nu), \smooth(\mu,\nu)) \leq \wass(\mu,\nu) \leq 3\maxarg{\smooth(\mu, \nu), \sqrt[3]{\smooth(\mu, \nu)\sqrt{2}\,\Earg{\twonorm{G}}^{2}}}.
\ealigns
\end{lemma}
\begin{proof}
The first inequality follows directly from the inclusions $\blset \subset \wassset$ and $\smoothset \subset \wassset$.

To establish the second, we fix $h\in\wassset$ and $t > 0$ and define the smoothed function
\[
	h_t(x) = \int_{\reals^d} h(x + t \pvar) \phi(\pvar) d\pvar \qtext{for each} x \in \reals^d,
\]
where $\phi$ is the density of a vector of $d$ independent standard normal variables.
We first show that $h_t$ is a close approximation to $h$ when $t$ is small.
Specifically, if $X \in\reals^d$ is an integrable random vector, independent of $G$, then,
by the Lipschitz assumption on $h$,
\baligns
 	|\Earg{h(X) - h_t(X)}|
		= |\Earg{h(X) - h(X+tG)}|
		\leq t\Earg{\twonorm{G}}.
\ealigns

We next show that the derivatives of $h_t$ are bounded.
Fix any $x\in\reals^d$.
Since $h$ is Lipschitz, it admits a weak gradient, $\grad h$, bounded uniformly by 1 in $\twonorm{\cdot}$.
We alternate differentiation and integration by parts to develop the representations
\baligns
	\grad h_t(x) 
		&= \int_{\reals^d} \grad h(x + t\pvar) \phi(\pvar) d\pvar 
		= \frac{1}{t} \int_{\reals^d} \pvar h(x + t\pvar) \phi(\pvar) d\pvar, \\
	\grad^2 h_t(x)
		&= \frac{1}{t} \int_{\reals^d} \grad h(x + t\pvar)\pvar^\top \phi(\pvar) d\pvar 
		= \frac{1}{t^2} \int_{\reals^d} (\pvar \pvar^\top - I) h(x + t\pvar) \phi(\pvar) d\pvar, \qtext{and} \\
	\grad^3 h_t(x)[v]
		&= \frac{1}{t^2} \int_{\reals^d} \grad h(x + t\pvar) v^\top(\pvar \pvar^\top - I)\phi(\pvar) d\pvar
\ealigns
for each $v\in\reals^d$.
The uniform bound on $\grad h$ now yields $M_1(h_t) \leq 1$,
\baligns
	M_2(h_t) %
		&\leq \frac{1}{t}\sup_{v\in\reals^d: \twonorm{v}=1} \int_{\reals^d} |\inner{\pvar}{v}|\phi(z)d\pvar 
		= \frac{1}{t}\sqrt{\frac{2}{\pi}}\sup_{v\in\reals^d: \twonorm{v}=1} \twonorm{v} 
		= 	\frac{1}{t}\sqrt{\frac{2}{\pi}}, \qtext{and}	\\
	M_3(h_t) %
		&\leq \frac{1}{t^2}\sup_{v,w\in\reals^d: \twonorm{v}=\twonorm{w}=1} \int_{\reals^d} |v^\top(\pvar \pvar^\top - I)w| \phi(z) d\pvar \\
		&\leq \frac{1}{t^2}\sup_{v,w\in\reals^d: \twonorm{v}=\twonorm{w}=1} \sqrt{\int_{\reals^d} |v^\top(\pvar \pvar^\top - I)w|^2 \phi(z) d\pvar} \\
		&= \frac{1}{t^2}\sup_{v,w\in\reals^d: \twonorm{v}=\twonorm{w}=1} \sqrt{\inner{v}{w}^2 + \twonorm{v}^2\twonorm{w}^2} 
		\leq \frac{\sqrt{2}}{t^2}. %
\ealigns
In the final equality we have used the fact that $\inner{v}{\PVAR}$ and $\inner{w}{\PVAR}$ are jointly normal with zero mean and covariance 
$\Sigma = \begin{bmatrix}\twonorm{v}^2&\inner{v}{w}\\ \inner{v}{w} &\twonorm{w}^2\end{bmatrix}$, so that the product $\inner{v}{\PVAR}\inner{w}{\PVAR}$
has the distribution of the off-diagonal element of the Wishart distribution~\cite{Wishart1928} with scale $\Sigma$ and $1$ degree of freedom.

We can now develop a bound for $\wass$ using our smoothed functions. 
Let
\[
\textstyle
b_t 
	\defeq \maxarg{1, \frac{1}{t} \sqrt{\frac{2}{\pi}}, \frac{\sqrt{2}}{t^2}} 
	= \maxarg{1, \frac{\sqrt{2}}{t^2}} 
\]
represent the maximum derivative bound of $h_t$, and select $X\sim\mu$ and $\PVAR\sim\nu$ to satisfy $\wass(\mu, \nu) = \Earg{\twonorm{X-\PVAR}}$.
If we let $c = \sqrt[3]{\smooth(\mu, \nu)\sqrt{2}\,\Earg{\twonorm{G}}^{2}}$, we then have
\baligns
	\wass(\mu, \nu) 
		&\leq \inf_{t>0}\sup_{h\in \wassset} |\Esubarg{\mu}{h(X) - h_t(X)}| + |\Esubarg{\nu}{h(\PVAR) - h_t(\PVAR)}| + |\Esubarg{\mu}{h_t(X)} - \Esubarg{\nu}{h_t(\PVAR)}| \\
		&\leq \inf_{t>0} 2t\Earg{\twonorm{G}} + b_t\smooth(\mu, \nu) 
		\leq 2 c + \maxarg{\smooth(\mu, \nu), c} 
		\leq 3\maxarg{\smooth(\mu, \nu), c},
\ealigns
where we have chosen $t = \sqrt[3]{\smooth(\mu,\nu)\sqrt{2}/\Earg{\twonorm{G}}}$ to achieve the third inequality.
\end{proof}
\begin{remark}
While \lemref{smooth-wass} targets Lipschitz test functions, comparable results can be obtained for non-smooth functions, like the indicators of convex
sets, by adapting the smoothing technique of \cite[Lem.~2.1]{Bentkus2003}. 
\end{remark}
\subsection{Example application to Bayesian logistic regression}
Before turning to the proof of \thmref{concave-stein-factors}, we illustrate a practical application to measuring the quality
of Monte Carlo or cubature sample points in Bayesian inference.
Consider the Bayesian logistic regression posterior density~\cite[see, e.g.,][]{GelmanCaStDuVeRu2014}
\baligns
p(\beta) \propto \underbrace{\exp{-\twonorm{\beta}^2/(2\sigma^2)}}_{\text{multivariate Gaussian prior}}\  \underbrace{\textprod_{l=1}^L {e^{y_l\inner{\beta}{v_l}}}{/(1+e^{\inner{\beta}{v_l}})}}_{\text{logistic regression likelihood}}
\ealigns
based on $L$ observed datapoints $(v_l, y_l)$ and a known prior hyperparameter $\sigma^2 > 0$.
In this standard model of binary classification, 
$\beta\in \R^d$ represents our inferential target, an unknown parameter vector with a multivariate Gaussian prior;
$y_l \in \{0,1\}$ is the class label of the $l$-th observed datapoint;
and $v_l\in \R^d$ is an associated vector of covariates. 

Since the normalizing constant of $p$ is unknown, it is common practice to approximate expectations $\int h(\beta) p(\beta) d\beta$ under $p$ with
sample estimates, $\frac{1}{n} \sum_{i=1}^n h(\beta_i)$, based on sample points $\beta_i \in \R^d$ drawn from a Markov chain or a cubature rule \cite{GelmanCaStDuVeRu2014}.
\thmref{concave-stein-factors} furnishes a way to uniformly bound the error of this approximation,
$|\frac{1}{n} \sum_{i=1}^n h(\beta_i) - \int h(\beta) p(\beta) d\beta|$,
for all sufficiently smooth functions $h$.

Concretely, we have, for all $\ell_2$ unit vectors $u_1,u_2,u_3,u_4 \in \R^d$,%
\baligns
	u_1^\top\Hess \log p(\beta)u_1 
		&=  -1/\sigma^2 - \sum_{l=1}^L  \frac{e^{\inner{\beta}{v_l}}}{(1+e^{\inner{\beta}{v_l}})^2} \inner{v_l}{u_1}^2 
		\leq -1/\sigma^2, \\  
	\grad^3 \log p(\beta)[u_1,u_2,u_3] &= - \sum_{l=1}^L  \frac{e^{\inner{\beta}{v_l}}(1-e^{\inner{\beta}{v_l}})}{(1+e^{\inner{\beta}{v_l}})^3} \prod_{m=1}^3\inner{v_l}{u_m}
		\leq \frac{\sum_{l=1}^L\twonorm{v_l}^3}{6\sqrt{3}},
		\text{ and} \\ 
	\grad^4 \log p(\beta)[u_1,u_2,u_3,u_4] &= - \sum_{l=1}^L  \frac{4e^{2\inner{\beta}{v_l}}-e^{3\inner{\beta}{v_l}}-e^{\inner{\beta}{v_l}}}{(1+e^{\inner{\beta}{v_l}})^4} \prod_{m=1}^4\inner{v_l}{u_m}
		\leq \frac{\sum_{l=1}^L\twonorm{v_l}^4}{8}.
\ealigns
Hence, \thmref{concave-stein-factors} applies with $k = 1/\sigma^2, L_3 = \frac{\sum_{l=1}^L\twonorm{v_l}^3}{6\sqrt{3}},$ and $L_4 = \frac{\sum_{l=1}^L\twonorm{v_l}^4}{8}$.
We may now plug the associated Stein factors 
\[\textstyle
(c_1,c_2,c_3) \defeq \left(2\sigma^2, \frac{\sigma^4\sum_{l=1}^L\twonorm{v_l}^3}{3\sqrt{3}} + \sigma^2, \frac{\sigma^6(\sum_{l=1}^L\twonorm{v_l}^3)^2}{18} + \frac{\sigma^4\sum_{l=1}^L\twonorm{v_l}^4}{8}+\frac{\sigma^4\sum_{l=1}^L\twonorm{v_l}^3}{2\sqrt{3}} + \frac{2\sigma^2}{3}\right)
\]
into the non-uniform graph Stein discrepancy of \cite{GorhamMa15} to obtain a computable
upper bound on $\smooth(Q,P)$ or $\wass(Q,P)$ for any discrete probability measure $Q = \frac{1}{n}\sum_{i=1}^n \delta_{\beta_i}$.

\section{Proof of {\thmref{concave-stein-factors}}}
\label{sec:concave-stein-factors-proof}
Before tackling the main proof, we will establish a series of useful lemmas.
We will make regular use of the following well-known Lipschitz property: %
\balign
M_k(h) &= \sup_{x \in \reals^d} \opnorm{\grad^k h(x)} \qtext{for all} h\in C^k(\reals^d)  \qtext{and} \text{each integer $k \geq 1$.} \label{eqn:lipk} 
\ealign
\subsection{Properties of overdamped Langevin diffusions}
Our first lemma enumerates several properties of the overdamped Langevin diffusion that will prove useful in the proofs to follow.
\begin{lemma}[Overdamped Langevin properties] \label{lem:langevin-properties}
	If $\log p \in C^2(\reals^d)$ is strongly concave, then the overdamped Langevin diffusion $(\process{t}{x})_{t\geq0}$
	with infinitesimal generator \eqref{eqn:langevin-generator} and $\process{0}{x}=x$
	is well-defined for all times $t\in [0,\infty)$, has stationary distribution $P$, and 
	satisfies \emph{strong continuity} on $L = \{f \in C^0(\reals^d) : \frac{|f(x)|}{1+\twonorm{x}^2}\to 0 \text{ as } \twonorm{x}\to \infty\}$
	with norm $\norm{f}_L \defeq \sup_{x\in\reals^d} \frac{|f(x)|}{1+\twonorm{x}^2}$,
	that is,	$\norm{\Earg{f(\process{t}{\cdot})} - f}_L \to 0$ as $t \to 0^+$ for all $f \in L$.
\end{lemma}
\begin{proof}
Consider the Lyapunov function $V(x) = \twonorm{x}^2+1$.
The strong log-concavity of $p$, the Cauchy-Schwarz inequality, and the arithmetic-geometric mean inequality imply that
\baligns
	&\genarg{V}{x} 
		= \inner{x}{\grad \log p(x)} + d 
		= \inner{x}{\grad \log p(x)-\grad \log p(0)} + \inner{x}{\grad \log p(0)} + d \\
		&\leq -k\twonorm{x}^2 + \twonorm{x}\twonorm{\grad \log p(0)} + d
		\leq \left(\frac{1}{2}-k\right)\twonorm{x}^2 + \frac{1}{2}\twonorm{\grad \log p(0)}^2 + d  
		\leq k' V(x)
\ealigns
for some constants $k, k' \in \reals$.
Since $\log p$ is locally Lipschitz, 
\cite[Thm.~3.5]{Khasminskii11} implies that the diffusion 
$(\process{t}{x})_{t\geq0}$ is well-defined, %
and \cite[Thm.~2.1]{RobertsTw96} guarantees that $P$ is a stationary distribution.
The argument of \cite[Prop. 15]{GorhamDuVoMa16} with \cite[Thm.~3.5]{Khasminskii11} substituted for \cite[Thm.~3.4]{Khasminskii11} and \cite[Sec. 5, Cor. 1.2]{Friedman1975Stoch} now yields strong continuity.
\end{proof}

\subsection{High-order weighted difference bounds}
A second, technical lemma bounds the growth of weighted smooth function differences in terms of the proximity of function arguments.
The result will be used to characterize the smoothness of $\process{t}{x}$ as a function of the starting point $x$ (\lemref{synchronous}) and, ultimately, to establish the smoothness of $u_h$ (\thmref{concave-stein-factors}).
\begin{lemma}[High-order weighted difference bounds]\label{lem:differences}
Fix any weights $\lambda,\lambda' >0$ and any vectors $x,y,z,w,x',y',z',w'\in\reals^d$.
If $h \in C^2(\reals^d)$, then
\balign
&|\lambda (h(x) - h(y)) - \lambda'(h(x') - h(y')) - \inner{\grad h(y)}{\lambda(x - y) - \lambda'(x' - y')}| \notag\\
 &\leq \texthalf M_2(h) (2\lambda'\twonorm{y-y'}\twonorm{x'-y'}+\lambda\twonorm{x-y}^2 + \lambda'\twonorm{x'-y'}^2).\label{eqn:second-order-diff}
\ealign
Moreover, if $h \in C^3(\reals^d)$, then
\balign
&|\lambda(h(x) - h(y) - (h(z) - h(w))) -\lambda'(h(x') - h(y') - (h(z') - h(w'))) \notag\\
& - \inner{\grad h(z)}{\lambda(x - y - (z - w)) - \lambda'(x' - y' - (z' -w'))}| \label{eqn:third-order-diff}\\
	&\leq M_2(h)\left[\twonorm{y'-x'}\twonorm{\lambda(z - x) - \lambda'(z' - x')}\right. \notag\\
	&+ \lambda'\twonorm{z-z'}\twonorm{x' - y'- (z'-w')}+\lambda\twonorm{z-x}\twonorm{(y-x) - (y'-x')} \notag\\
	&+ \left.\texthalf(\lambda\twonorm{x-y-(z-w)}\twonorm{x-y+z-w}  
	+ \lambda'\twonorm{x'-y'-(z'-w')}\twonorm{x'-y'+z'-w'})\right] \notag\\
   &+ M_3(h) \left[\texthalf \twonorm{y'-x'} (2\lambda'\twonorm{x-x'}\twonorm{z'-x'}+\lambda\twonorm{z-x}^2 +\lambda'\twonorm{z'-x'}^2)\right. \notag\\
	&+ \texthalf(\lambda\twonorm{z-x}\twonorm{y-x}^2 + \lambda'\twonorm{z'-x'}\twonorm{y'-x'}^2) \notag\\
	&+ \left.\textfrac{1}{6}(\lambda\twonorm{w-z}^3+\lambda\twonorm{y-x}^3+\lambda'\twonorm{w'-z'}^3+\lambda'\twonorm{y'-x'}^3)\right]. \notag
\ealign
\end{lemma}
\begin{proof}
To establish the second-order difference bound \eqref{eqn:second-order-diff}, we first apply Taylor's theorem with mean-value remainder
to $h(x) - h(y)$ and $h(x') - h(y')$ to obtain
\baligns
&\lambda(h(x) - h(y)) - \lambda'(h(x') - h(y')) - \inner{\grad h(y)}{\lambda(x-y)-\lambda'(x'-y')} \\ 
=\ &\lambda'\inner{\grad h(y)-\grad h(y')}{x'-y'} + (\lambda\inner{\Hess h(\zeta)(x-y)}{x-y} - \lambda'\inner{\Hess h(\zeta')(x'-y')}{x'-y'})/2 
\ealigns
for some $\zeta,\zeta'\in \reals^d$. %
Cauchy-Schwarz, the definition of the operator norm, and the Lipschitz gradient relation \eqref{eqn:lipk} now yield
the advertised conclusion \eqnref{second-order-diff}.

To derive the third-order difference bound \eqref{eqn:third-order-diff}, we apply Taylor's theorem with mean-value remainder
to $h(w) - h(z)$, $h(y) - h(x)$, $h(w') - h(z')$, and $h(y') - h(x')$ to write
\balign
&|\lambda(h(x) - h(y) - (h(z) - h(w))) -\lambda'(h(x') - h(y') - (h(z') - h(w'))) \notag\\
& -\inner{\grad h(z)}{\lambda(x - y - (z - w)) - \lambda'(x' - y' - (z' -w'))}| \label{eqn:to-bound} \\
  &= |\lambda'\inner{\grad h(z)-\grad h(z')}{x' - y' -(z'-w')} + \lambda\inner{\grad h(z)-\grad h(x)}{(y-x) - (y'-x')} \notag\\
	&	+\inner{\lambda(\grad h(z)-\grad h(x))- \lambda'(\grad h(z')-\grad h(x'))}{y'-x'} \notag\\
	&+ \lambda\inner{\Hess h(z)(w-z)}{w-z}/2 - \lambda\inner{\Hess h(x)(y-x)}{y-x}/2  \notag\\
	&- \lambda'\inner{\Hess h(z')(w'-z')}{w'-z'}/2 + \lambda'\inner{\Hess h(x')(y'-x')}{y'-x'}/2 \notag\\
	&+ \lambda\grad^3 h(\zeta'')[w-z,w-z,w-z]/6- \lambda\grad^3 h(\zeta'''')[y-x,y-x,y-x]/6 \notag\\
	&-\lambda'\grad^3 h(\zeta''')[w'-z',w'-z',w'-z']/6+\lambda'\grad^3 h(\zeta''''')[y'-x',y'-x',y'-x']/6 \notag|
\ealign
for some $\zeta'',\zeta''',\zeta'''',\zeta'''''\in\reals^d$.
We will bound each line in this expression in turn. %
First we see, by Cauchy-Schwarz and the Lipschitz property \eqref{eqn:lipk}, that
\begin{gather*}
|\lambda'\inner{\grad h(z)-\grad h(z')}{x' - y' -(z'-w')} + \lambda\inner{\grad h(z)-\grad h(x)}{(y-x) - (y'-x')}|\\
	\leq M_2(h)(\lambda'\twonorm{z-z'}\twonorm{x' - y'- (z'-w')}+\lambda\twonorm{z-x}\twonorm{(y-x) - (y'-x')}).
\end{gather*}
Next, we invoke our second-order difference bound \eqref{eqn:second-order-diff} on the $C^2(\reals^d)$ function $x\mapsto \inner{\grad h(x)}{y'-x'}$, 
apply the Cauchy-Schwarz inequality, and use the definition of the operator norm to conclude that
\baligns
&|\inner{\lambda(\grad h(z)-\grad h(x))- \lambda'(\grad h(z')-\grad h(x'))}{y'-x'}| \\
	&\leq M_2(h)\twonorm{y'-x'}\twonorm{\lambda(z - x) - \lambda'(z' - x')} \\
   &+ \frac{1}{2}M_3(h) \twonorm{y'-x'} (2\lambda'\twonorm{x-x'}\twonorm{z'-x'}+\lambda\twonorm{z-x}^2 +\lambda'\twonorm{z'-x'}^2).
\ealigns
To bound the subsequent line, we note that Cauchy-Schwarz, the definition of the operator norm, and the Lipschitz property \eqref{eqn:lipk} imply that
\baligns
&|\inner{\Hess h(z)(w-z)}{w-z} - \inner{\Hess h(x)(y-x)}{y-x}| \\
	&=|\inner{\Hess h(z)(w-z+y-x)}{x-y-(z-w)} +\inner{(\Hess h(z) - \Hess h(x))(y-x)}{y-x}| \\
	&\leq M_2(h) \twonorm{x-y-(z-w)}\twonorm{x-y+z-w} 
	+M_3(h)\twonorm{z-x}\twonorm{y-x}^2.
\ealigns
Similarly, 
\baligns
&|\inner{\Hess h(z')(w'-z')}{w'-z'} - \inner{\Hess h(x')(y'-x')}{y'-x'}| \\
	&\leq M_2(h) \twonorm{x'-y'-(z'-w')}\twonorm{x'-y'+z'-w'} 
	+M_3(h)\twonorm{z'-x'}\twonorm{y'-x'}^2.
\ealigns
Finally, Cauchy-Schwarz and the definition of the operator norm give
\baligns
&|\lambda\grad^3 h(\zeta'')[w-z,w-z,w-z] - \lambda\grad^3 h(\zeta'''')[y-x,y-x,y-x] \\
	&-\lambda'\grad^3 h(\zeta''')[w'-z',w'-z',w'-z'] +\lambda'\grad^3 h(\zeta''''')[y'-x',y'-x',y'-x']| \\
	&\leq M_3(h)(\lambda\twonorm{w-z}^3+\lambda\twonorm{y-x}^3+\lambda'\twonorm{w'-z'}^3+\lambda'\twonorm{y'-x'}^3).
\ealigns
Bounding the third-order difference \eqref{eqn:to-bound} in terms of these four estimates yields \eqref{eqn:third-order-diff}.
\end{proof}

\subsection{Synchronous coupling lemma}
Our proof of \thmref{concave-stein-factors} additionally rests upon a series of coupling inequalities which serve to characterize the smoothness of $\process{t}{x}$ as a function of $x$.
The couplings espoused in the lemma to follow are termed \emph{synchronous}, because the same Brownian motion is used to drive each process.

\begin{lemma}[Synchronous coupling inequalities] \label{lem:synchronous}
Suppose that $\log p \in C^4(\reals^d)$ is $k$-strongly concave with $M_3(\log p)\leq L_3$ and $M_4(\log p) \leq L_4$.
Fix a $d$-dimensional Wiener process $(W_t)_{t\geq 0}$, 
any vectors $x,x',v,v' \in \reals^d$ with $\twonorm{v}=\twonorm{v'}=1$, and any weights $\eps, \eps', \eps'' > 0$, and
define the growth factors
\balign\label{eqn:growth-factors}
f_1(x,x',\eps,\eps',\eps'') &\defeq \twonorm{x-x'}+(\eps''+\eps')/2+\eps(3+\eps/\eps''+\eps/\eps'+\twonorm{x-x'}/\eps')/3 \qtext{and}\notag\\
f_2(x,x',\eps,\eps',\eps'') &\defeq \twonorm{x-x'}+3(\eps''+\eps')/2+\eps(3+\eps/\eps''+\eps/\eps')/3. 
\ealign

For each starting point of the form $z+b'v' +b v$ with $z\in\{x,x'\}$, $b' \in \{0,\eps',\eps''\}$, and $b\in\{0,\eps\}$,
consider an overdamped Langevin diffusion $(\process{t}{z+b'v' +b v})_{t\geq0}$ solving the stochastic differential equation
\balign\label{eqn:synchronous-sde}
	d\process{t}{z+b'v'+bv} &= \frac{1}{2}\grad\log p(\process{t}{z+b'v'+bv}) dt + dW_t\qtext{with} \process{0}{z+b'v'+bv}=z+b'v'+bv, 
\ealign
and define the differenced processes
\baligns
V_t &\defeq (\process{t}{x'+\eps'' v'} - \process{t}{x'})/\eps'' - (\process{t}{x+\eps' v'} - \process{t}{x})/\eps' \qtext{and}\\
U_t 
	&\defeq {\process{t}{x'+\eps'' v'+\eps v} - \process{t}{x'+\eps''v'} -(\process{t}{x'+\eps v} - \process{t}{x'})}{/\eps\eps''} \\
	&-{\process{t}{x+\eps' v'+\eps v} - \process{t}{x+\eps'v'} -(\process{t}{x+\eps v} - \process{t}{x})}{/\eps\eps'}.
\ealigns
These coupled processes almost surely satisfy the synchronous coupling bounds,
\balign
	e^{kt/2}&\twonorm{\process{t}{x+\eps v} - \process{t}{x}} 
		\leq \eps,\label{eqn:contract1} \\
	e^{kt/2}&\twonorm{V_t}
		\leq \frac{L_3}{k}(\twonorm{x-x'}+(\eps''+\eps')/2),\qtext{and}\label{eqn:contract2} \\
	e^{kt/2}&\twonorm{U_t} 
		\leq \frac{3L_3^2}{k^2}f_1(x,x',\eps,\eps',\eps'')
		+\frac{L_4}{2k}f_2(x,x',\eps,\eps',\eps''), \label{eqn:contract3}
\ealign
the second-order differenced function bound,
\balign \label{eqn:function-contract2}
	&(h_2(\process{t}{x'+\eps''v'}) - h_2(\process{t}{x'}))/\eps'' - (h_2(\process{t}{x+\eps' v'}) - h_2(\process{t}{x}))/\eps' \\
		\leq &\left(M_1(h_2)\frac{L_3}{k}e^{-kt/2} +  M_2(h_2) e^{-kt}\right)(\twonorm{x-x'}+(\eps''+\eps')/2),\notag
\ealign
and the third-order differenced function bound,
\balign \label{eqn:function-contract3}
	&{(h_3(\process{t}{x'+\eps'' v'+\eps v}) - h_3(\process{t}{x'+\eps''v'}) -(h_3(\process{t}{x'+\eps v}) - h_3(\process{t}{x'})))}{/(\eps\eps'')} \notag\\
		&-{(h_3(\process{t}{x+\eps' v'+\eps v}) - h_3(\process{t}{x+\eps'v'}) -(h_3(\process{t}{x+\eps v}) - h_3(\process{t}{x})))}{/(\eps\eps')} \\
	\leq &\,\left(M_1(h_3)\frac{3L_3^2}{k^2}e^{-kt/2} +  M_2(h_3) \frac{3L_3}{k}e^{-kt}\right) f_1(x,x',\eps,\eps',\eps'') \notag\\
	+ &\left(M_1(h_3)\frac{L_4}{2k}e^{-kt/2} + M_3(h_3) e^{-3kt/2}\right) f_2(x,x',\eps,\eps',\eps'')\notag%
\ealign
for each $t\geq 0$, $h_2 \in C^2(\reals^d)$, and $h_3 \in C^3(\reals^d)$.
\end{lemma}
\begin{proof}
By \lemref{langevin-properties}, each process $(\process{t}{z+b'v' +b v})_{t\geq0}$ with $z\in\{x,x'\}$, $b' \in \{0,\eps',\eps''\}$, and $b\in\{0,\eps\}$ is well-defined for all times $t \in[0,\infty)$.
The first-order bound \eqref{eqn:contract1} is well known, and a concise proof can be found in \cite{CattiauxGu14}.
\paragraph{Second-order bounds}
To establish the second conclusion \eqref{eqn:contract2}, we consider the \Ito process of second-order differences
\baligns
V_t 
	&= \frac{1}{2}\int_0^t \frac{\grad\log p(\process{s}{x'+\eps'' v'}) - \grad\log p(\process{s}{x'})}{\eps''} - \frac{\grad\log p(\process{s}{x+\eps' v'}) - \grad\log p(\process{s}{x})}{\eps'}\, ds 
\ealigns
and apply \Ito's lemma to the mapping $(t,w)\mapsto e^{kt/2}\twonorm{w}$.  This yields
\baligns
&e^{kt/2}\twonorm{V_t} 
	= e^{0}\twonorm{V_0} + \int_0^t ke^{ks}\twonorm{V_s} + e^{ks}\deriv{}{s}\twonorm{V_s}\, ds \\
	&= \int_0^t \frac{e^{ks/2}}{2\twonorm{V_s}} (k\twonorm{V_s}^2 \\
	&+ \binner{V_s}{(\grad\log p(\process{s}{x'+\eps'' v'}) - \grad\log p(\process{s}{x'}))/\eps'' - (\grad\log p(\process{s}{x+\eps' v'}) - \grad\log p(\process{s}{x}))/\eps'}) ds.
\ealigns
Fix a value $s\in[0,t]$. 
For any $h_2\in C^2(\reals^d)$, the \lemref{differences} second-order difference inequality \eqref{eqn:second-order-diff},
the first order coupling bound \eqref{eqn:contract1}, Cauchy-Schwarz, and the Lipschitz identity \eqref{eqn:lipk} together give the estimates
\balign
&(h_2(\process{s}{x'+\eps''v'}) - h_2(\process{s}{x'}))/\eps'' - (h_2(\process{s}{x+\eps' v'}) - h_2(\process{s}{x}))/\eps' \notag \\
	\leq\ &\inner{\grad h_2(\process{s}{x'})}{V_s} 
	+\frac{1}{2} M_2(h_2) (2\twonorm{\process{s}{x'} - \process{s}{x}}\twonorm{\process{s}{x+\eps' v'}-\process{s}{x}}/\eps'\notag\\
	&+\twonorm{\process{s}{x'+\eps'' v'}-\process{s}{x'}}^2/\eps''
	+\twonorm{\process{s}{x+\eps' v'}-\process{s}{x}}^2/\eps') \notag\\
	\leq\ &\inner{\grad h_2(\process{s}{x'})}{V_s}	+  M_2(h_2) e^{-ks}(\twonorm{x-x'}+(\eps''+\eps')/2)\label{eqn:pre-contract2} \\
	\leq\ &M_1(h_2)\twonorm{V_s}	+  M_2(h_2) e^{-ks}(\twonorm{x-x'}+(\eps''+\eps')/2). \label{eqn:pre-function-contract2}
\ealign
Applying the estimate \eqref{eqn:pre-contract2} to the $C^2(\reals^d)$ function $h_2(z)=\inner{V_s}{\grad\log p(z)}$ with
$
 M_2(h_2)  = \sup_{z\in\reals^d}\opnorm{\grad^3\log p(z)[V_s]} \leq L_3 \twonorm{V_s},
$ yields
\baligns
&\inner{V_s}{(\grad\log p(\process{s}{x'+\eps'' v'}) - \grad\log p(\process{s}{x'}))/\eps'' - (\grad\log p(\process{s}{x+\eps' v'}) - \grad\log p(\process{s}{x}))/\eps'} \\
	\leq\ &\inner{V_s}{\Hess \log p(\process{s}{x'})V_s} 
		+ L_3\twonorm{V_s}e^{-ks}(\twonorm{x-x'}+(\eps''+\eps')/2) \\
	\leq &-k\twonorm{V_s}^2
		+ L_3\twonorm{V_s}e^{-ks}(\twonorm{x-x'}+(\eps''+\eps')/2),
\ealigns
where, to achieve the second inequality, we used the $k$-strong log-concavity of $p$.
Now we may derive the second-order synchronous coupling bound \eqref{eqn:contract2}, since
\baligns
e^{kt/2}\twonorm{V_t} 
	&\leq \frac{L_3}{2}(\twonorm{x-x'}+(\eps''+\eps')/2)\int_0^t e^{-ks/2} ds 
	\leq \frac{L_3}{k}(\twonorm{x-x'}+(\eps''+\eps')/2).
\ealigns
Applying the synchronous coupling bound \eqref{eqn:contract2} to the estimate \eqref{eqn:pre-function-contract2} finally delivers the second-order differenced function bound
 \eqref{eqn:function-contract2}.
\paragraph{Third-order bounds}
To establish the third conclusion \eqref{eqn:contract3}, we consider the \Ito process of third-order differences 
\baligns
U_t 
	= \frac{1}{2}\int_0^t&\frac{\grad\log p(\process{s}{x'+\eps'' v'+\eps v}) - \grad\log p(\process{s}{x'+\eps''v'}) 
		-(\grad\log p(\process{s}{x'+\eps v})-\grad\log p(\process{s}{x'}))}{\eps\eps''} \\
	- &\frac{\grad\log p(\process{s}{x+\eps' v'+\eps v}) - \grad\log p(\process{s}{x+\eps'v'}) 
		-(\grad\log p(\process{s}{x+\eps v})-\grad\log p(\process{s}{x}))}{\eps\eps'}\ ds
\ealigns
and invoke \Ito's lemma once more for the mapping $(t,w)\mapsto e^{kt/2}\twonorm{w}$.
This produces
\baligns
&e^{kt/2}\twonorm{U_t}
= e^{0}\twonorm{U_0} + \int_0^t ke^{ks}\twonorm{U_s} + e^{ks}\deriv{}{s}\twonorm{U_s}\, ds \\
&= \int_0^t \frac{e^{ks/2}}{2\twonorm{U_s}}\big(k\twonorm{U_s}^2 \\
&+ \textfrac{1}{\eps\eps''}\binner{U_s}{\grad\log p(\process{s}{x'+\eps'' v'+\eps v}) - \grad\log p(\process{s}{x'+\eps''v'}) 
		-(\grad\log p(\process{s}{x'+\eps v})-\grad\log p(\process{s}{x'}))} \\
&- \textfrac{1}{\eps\eps'}\binner{U_s}{\grad\log p(\process{s}{x+\eps' v'+\eps v}) - \grad\log p(\process{s}{x+\eps'v'}) 
		-(\grad\log p(\process{s}{x+\eps v})-\grad\log p(\process{s}{x}))}\big) ds.
\ealigns
Fix a value $s\in[0,t]$, and introduce the shorthand $c_1 \defeq f_1(x,x',\eps,\eps',\eps'')$ and $c_2 \defeq f_2(x,x',\eps,\eps',\eps'')$.
For any $h_3\in C^3(\reals^d)$, the \lemref{differences} third-order difference inequality \eqref{eqn:third-order-diff}, the coupling bounds  \eqref{eqn:contract1} and \eqref{eqn:contract2}, 
Cauchy-Schwarz, and the Lipschitz identity \eqref{eqn:lipk} together imply the estimates 
\balign
	&{(h_3(\process{s}{x'+\eps'' v'+\eps v}) - h_3(\process{s}{x'+\eps''v'}) -(h_3(\process{s}{x'+\eps v}) - h_3(\process{s}{x'})))}{/(\eps\eps'')} \notag\\
		&-{(h_3(\process{s}{x+\eps' v'+\eps v}) - h_3(\process{s}{x+\eps'v'}) -(h_3(\process{s}{x+\eps v}) - h_3(\process{s}{x})))}{/(\eps\eps')} \notag\\
	\leq &\,\inner{\grad h_3(\process{s}{x'+\eps''v'})}{U_s} 
		+ M_2(h_3) \frac{L_3}{k}e^{-ks}(2\twonorm{x-x'}+\twonorm{x-x'+(\eps'-\eps'')v'})\notag\\
		&+  M_2(h_3) \frac{L_3}{k}e^{-ks} ((\eps''+\eps')/2+\eps(3+\eps/\eps''+\eps/\eps'+\twonorm{x-x'}/\eps'))\notag\\
		&+  M_3(h_3) e^{-3ks/2}(\twonorm{x-x'+(\eps'-\eps'')v'}+(\eps''+\eps')/2+\eps(3+\eps/\eps''+\eps/\eps')/3). \notag\\
	\leq &\,\inner{\grad h_3(\process{s}{x'+\eps''v'})}{U_s} +  M_2(h_3) \frac{3L_3}{k}e^{-ks} c_1
		+  M_3(h_3) e^{-3ks/2}c_2, \label{eqn:pre-contract3}\\
	\leq &\,M_1(h_3)\twonorm{U_s} +  M_2(h_3) \frac{3L_3}{k}e^{-ks} c_1
		+  M_3(h_3) e^{-3ks/2}c_2, \label{eqn:pre-function-contract3}
\ealign
where we have applied the triangle inequality to achieve \eqref{eqn:pre-contract3}.
Applying the bound \eqref{eqn:pre-contract3} to the thrice continuously differentiable function $h_3(z)=\inner{U_s}{\grad\log p(z)}$ with
$ %
M_2(h_3) = \sup_{z\in\reals^d}\opnorm{\grad^3\log p(z)[U_s]} 
	\leq L_3\twonorm{U_s}\text{ and } 
M_3(h_3) %
	\leq L_4\twonorm{U_s}
$ %
gives
\baligns
&(h_{3}(\process{s}{x'+\eps'' v'+\eps v}) - h_{3}(\process{s}{x'+\eps''v'}) 
		-(h_{3}(\process{s}{x'+\eps v})-h_{3}(\process{s}{x'})))/(\eps\eps'') \\
	&\quad- (h_{3}(\process{s}{x+\eps' v'+\eps v}) - h_{3}(\process{s}{x+\eps'v'}) 
		-(h_{3}(\process{s}{x+\eps v})-h_{3}(\process{s}{x})))/(\eps\eps') \\
	&\leq\ \inner{U_s}{\Hess \log p(\process{s}{x'+\eps''v'})U_s}  
		+ \twonorm{U_s}(\frac{3L_3^2}{k}e^{-ks} c_1
		+ L_4e^{-3ks/2}c_2). \notag \\
	&\leq -k\twonorm{U_s}^2 
		+ \twonorm{U_s}\frac{3L_3^2}{k}e^{-ks} c_1
		+ \twonorm{U_s}L_4e^{-3ks/2}c_2. \notag 
\ealigns
In the final line, we used the $k$-strong log-concavity of $p$.
Our efforts now yield \eqref{eqn:contract3} via
\baligns
e^{kt/2}\twonorm{U_t} 
	\leq \int_0^t\frac{3L_3^2}{2k}e^{-ks/2} c_1 + \frac{L_4}{2}e^{-ks}c_2ds
	\leq \frac{3L_3^2}{k^2} c_1 + \frac{L_4}{2k}c_2.
\ealigns
The third-order differenced function bound \eqref{eqn:function-contract3} then follows by applying the third-order synchronous coupling bound \eqref{eqn:contract3} to the estimate \eqref{eqn:pre-function-contract3}.
\end{proof}
\subsection{Proof of \thmref{concave-stein-factors}}
By \lemref{langevin-properties}, for each $x\in\reals^d$, the overdamped Langevin diffusion $(\process{t}{x})_{t\geq0}$ is well-defined with stationary distribution $P$.
Moreover, for each $x\in\reals^d$, the diffusion $(\process{t}{x})_{t\geq0}$, by definition, satisfies
\[
	d\process{t}{x} = \frac{1}{2}\grad\log p(\process{t}{x}) dt + dW_t\qtext{with} \process{0}{x}=x, 
\]
for $(W_t)_{t\geq 0}$ a $d$-dimensional Wiener process.
In what follows, when considering the joint distribution of a finite collection of overdamped Langevin diffusions, we will assume that the diffusions are coupled in the manner of \lemref{synchronous}, so that each diffusion is driven by a shared $d$-dimensional Wiener process $(W_t)_{t\geq 0}$.

Fix any $x\in\reals^d$ and any $h \in C^3(\reals^d)$ with bounded first, second, and third derivatives.
We divide the remainder of our proof into five components, 
establishing that $u_h$ exists, $u_h$ is Lipschitz, $u_h$ has a Lipschitz gradient, $u_h$ has a Lipschitz Hessian,
and $u_h$ solves the Stein equation \eqref{eqn:stein-equation-generator}.

\paragraph{Existence of $u_h$}
To see that the integral representation of $u_h(x)$ is well-defined, note that
\baligns
&\int_0^\infty|\Esubarg{P}{h(\PVAR)} - \Earg{h(\process{t}{x})}|\ dt
	= \int_0^\infty\left|\int_{\reals^d} \Earg{h(\process{t}{y})} - \Earg{h(\process{t}{x})}\ p(y) dy\right|\ dt \\
	&\leq  M_1(h)  \int_0^\infty\int_{\reals^d} \Earg{\twonorm{\process{t}{y} - \process{t}{x}}}\ p(y) dy\ dt 
	\leq  M_1(h) \, \Esubarg{P}{\twonorm{Z-x}} \int_0^\infty e^{-kt/2}\ dt 
	< \infty.
\ealigns
The first relation uses the stationarity of $P$, the second uses the Lipschitz relation \eqref{eqn:lipk},
the third uses the first-order coupling inequality \eqref{eqn:contract1} of \lemref{synchronous},
and the last uses the fact that strongly log-concave distributions have subexponential tails and therefore finite moments of all orders~\cite[Lem. 1]{CuleSa10}.

\paragraph{Lipschitz continuity of $u_h$}
We next show that $u_h$ is Lipschitz.
Fix any vector $v\in \reals^d$, and consider the difference
\balign \notag
|u_h(x+v) - u_h(x)|
	&=\left| \int_0^\infty \Earg{h(\process{t}{x})-h(\process{t}{x+v})}\ dt\right|  \notag
	\leq  M_1(h) \int_0^\infty \Earg{\twonorm{\process{t}{x}-\process{t}{x+v}}}\ dt \\
	&\leq \twonorm{v}   M_1(h) \int_0^\infty e^{-kt/2}\ dt 
	= \frac{2}{k} \twonorm{v}  M_1(h) . \label{eqn:first-order-u-bound}
\ealign
The second relation is an application of the Lipschitz relation \eqref{eqn:lipk}, and the third applies the first-order coupling inequality \eqref{eqn:contract1} of \lemref{synchronous}.

\paragraph{Lipschitz continuity of $\grad u_h$}
To demonstrate that $u_h$ is differentiable with Lipschitz gradient, we first establish a weighted second-order difference inequality for $u_h$.

\begin{lemma}\label{lem:second-order-u}
For any vectors $x,x',v' \in \reals^d$ with $\twonorm{v'}=1$ and weights $\eps',\eps'' > 0$, 
\begin{gather}
\left|(u_h(x'+\eps''v')-u_h(x'))/\eps'' - (u_h(x+\eps' v')-u_h(x))/\eps'\right| \notag\\
	\leq (\twonorm{x-x'}+(\eps''+\eps')/2)\left( M_1(h) \frac{2L_3}{k^2}
		+  M_2(h) \frac{1}{k}\right). \label{eqn:second-order-u}
\end{gather}
\end{lemma}
\begin{proof}
We apply the \lemref{synchronous} second-order function coupling inequality \eqref{eqn:function-contract2} to obtain
\baligns
&\left|(u_h(x'+\eps''v')-u_h(x'))/\eps'' - (u_h(x+\eps' v')-u_h(x))/\eps'\right| \notag\\ 
	= &\left|\int_0^\infty \Earg{h(\process{t}{x'+\eps'v})-h(\process{t}{x'})}/\eps' - \Earg{h(\process{t}{x+\eps v})-h(\process{t}{x})}/\eps\ dt \right| \notag\\
	\leq & (\twonorm{x-x'} + (\eps' + \eps)/2)\int_0^\infty M_1(h) \frac{L_3}{k}e^{-kt/2}
		+  M_2(h) e^{-kt} \ dt.%
\ealigns
The desired bound follows by integrating the final expression.
\end{proof}

Now, fix any $x, v\in\reals^d$ with $\twonorm{v}=1$.
As a first application of the \lemref{second-order-u} second-order difference inequality \eqref{eqn:second-order-u}, we will demonstrate the existence of the directional derivative
\balign\label{eqn:directional-deriv-u}
	\grad_v u_h(x) \defeq \lim_{\epsilon \to 0} \frac{u_h(x+\epsilon v)-u_h(x)}{\epsilon}.
\ealign
Indeed, \lemref{second-order-u} implies that, for any integers $m,m' > 0$, 
\baligns
		&\left|m'(u_h(x+v/m')-u_h(x)) - m(u_h(x+v/m)-u_h(x))\right| \\
	\leq &\left(\frac{1}{2m}+\frac{1}{2m'}\right)\left( M_1(h) \frac{2L_3}{k^2} +  M_2(h) \frac{1}{k}\right).
\ealigns
Hence, the sequence $\left(\frac{u_h(x + v/m) - u_h(x)}{1/m}\right)_{m=1}^\infty$ is Cauchy, and the directional derivative \eqref{eqn:directional-deriv-u} exists.

To see that the directional derivative \eqref{eqn:directional-deriv-u} is also Lipschitz, fix any $v'\in\reals^d$, and consider the bound
\balign\label{eqn:second-order-u-bound}
&\left|\grad_v u_h(x+v') - \grad_v u_h(x)\right| 
	\leq \lim_{\epsilon \to 0}\left| \frac{u_h(x+\epsilon v+v')-u_h(x+v')}{\epsilon} -  \frac{u_h(x+\epsilon v)-u_h(x)}{\epsilon}\right| \notag\\ 
	\leq &\lim_{\epsilon \to 0}(\twonorm{v'}+\eps)\left( M_1(h) \frac{2L_3}{k^2}
		+  M_2(h) \frac{1}{k}\right) 
	= \twonorm{v'}\left( M_1(h) \frac{2L_3}{k^2}
		+  M_2(h) \frac{1}{k}\right),
\ealign
where the second inequality follows from \lemref{second-order-u}.
Since each directional derivative is Lipschitz continuous, we may conclude that $u_h$ is continuously differentiable 
with Lipschitz continuous gradient $\grad u_h$.
Our Lipschitz function deduction \eqref{eqn:first-order-u-bound} and the Lipschitz relation \eqref{eqn:lipk} additionally supply the uniform bound
$
M_1(u_h) %
	\leq \frac{2}{k}  M_1(h) . 
$

\paragraph{Lipschitz continuity of $\Hess u_h$}
To demonstrate that $\grad u_h$ is differentiable with Lipschitz gradient, we begin by establishing a weighted third-order difference inequality for $u_h$.

\begin{lemma}\label{lem:third-order-u}
Fix any vectors $x,x',v,v' \in \reals^d$ with $\twonorm{v}=\twonorm{v'}=1$ and weights $\eps,\eps',\eps'' > 0$, 
and define $f_1(x,x',\eps,\eps',\eps'')$ and $f_2(x,x',\eps,\eps',\eps'')$ as in \eqref{eqn:growth-factors} .  Then, 
\balign\label{eqn:third-order-u}
&|({u_h(x'+\eps'' v'+\eps v)-u_h(x'+\eps'' v')-(u_h(x'+\eps v)-u_h(x')))}{/(\eps\epsilon'')} \notag\\
		&- ({u_h(x+\eps' v'+\eps v)-u_h(x+\eps' v')-(u_h(x+\eps v)-u_h(x)))}{/(\eps\epsilon')}| \\
\leq & \left(M_1(h) \frac{6L_3^2}{k^3}+M_2(h) \frac{3L_3}{k^2}\right)f_1(x,x',\eps,\eps',\eps'')
		+  \left(M_1(h) \frac{L_4}{k^2}+M_3(h) \frac{2}{3k}\right)f_2(x,x',\eps,\eps',\eps'').\notag
\ealign
\end{lemma}
\begin{proof}
Introduce the shorthand $c_1 \defeq f_1(x,x',\eps,\eps',\eps'')$ and $c_2 \defeq f_2(x,x',\eps,\eps',\eps'')$.
We apply the \lemref{synchronous} third-order function coupling inequality \eqref{eqn:function-contract3} to the thrice continuously differentiable function $h$ to obtain
\baligns
&|({u_h(x'+\eps'' v'+\eps v)-u_h(x'+\eps'' v')-(u_h(x'+\eps v)-u_h(x')))}{/(\eps\epsilon'')} \notag\\
		&- ({u_h(x+\eps' v'+\eps v)-u_h(x+\eps' v')-(u_h(x+\eps v)-u_h(x)))}{/(\eps\epsilon')}| \\
	= &\bigg|\int_0^\infty \Earg{(h(\process{t}{x'+\eps'' v'+\eps v}) - h(\process{t}{x'+\eps''v'}) -(h(\process{t}{x'+\eps v}) - h(\process{t}{x'})))}{/(\eps\eps'')} \notag\\
		&-\Earg{(h(\process{t}{x+\eps' v'+\eps v}) - h(\process{t}{x+\eps'v'}) -(h(\process{t}{x+\eps v}) - h(\process{t}{x})))}{/(\eps\eps')}\ dt\bigg|\\
	\leq &\int_0^\infty
	    \left(M_1(h) \frac{3L_3^2}{k^2}e^{-kt/2} +M_2(h) \frac{3L_3}{k}e^{-kt}\right)c_1
		+ \left(M_1(h) \frac{L_4}{2k}e^{-kt/2}+  M_3(h) e^{-3kt/2}\right)c_2 \dt.	 
\ealigns
Integrating this final expression yields the advertised bound.
\end{proof}
Now, fix any $x, v, v'\in\reals^d$ with $\twonorm{v}=\twonorm{v'}=1$.
As a first application of the \lemref{third-order-u} third-order difference inequality \eqref{eqn:third-order-u}, we will demonstrate the existence of the second-order directional derivative
\balign\label{eqn:second-directional-deriv-u}
	\grad_{v'} \grad_v u_h(x) 
		&\defeq \lim_{\epsilon' \to 0} \frac{\grad_v u_h(x+\epsilon' v')- \grad_v u_h(x)}{\epsilon'} \\
		&= \lim_{\epsilon' \to 0} \lim_{\epsilon \to 0} \frac{u_h(x+\epsilon' v'+\epsilon v) - u_h(x+\epsilon v)-(u_h(x+\epsilon' v') - u_h(x))}{\epsilon\epsilon'}.\notag
\ealign
\lemref{third-order-u} guarantees that, for any integers $m,m' > 0$, 
\baligns
&\left|m'(\grad_v u_h(x+v'/m')-\grad_v u_h(x)) - m(\grad_v u_h(x+v'/m)-\grad_vu_h(x))\right| \\
	\leq &\lim_{\eps\to0} |m'(u_h(x+v'/m'+v\eps)-u_h(x+v'/m')-(u_h(x+v\eps)-u_h(x)))/\eps \\
		&\quad- m(u_h(x+v'/m+v\eps)-u_h(x+v'/m)-(u_h(x+v\eps) - u_h(x)))/\eps| \\
	\leq &\left(\frac{1}{m}+\frac{1}{m'}\right)\left( M_1(h) \left(\frac{3L_3^2}{k^3}+\frac{3L_4}{2k^2}\right)
		+  M_2(h) \frac{3L_3}{2k^2}
		+  M_3(h) \frac{1}{k}\right).
\ealigns
Hence, the sequence $\left(\frac{\grad_v u_h(x + v'/m) - \grad_v u_h(x)}{1/m}\right)_{m=1}^\infty$ is Cauchy, and the directional derivative \eqref{eqn:second-directional-deriv-u} exists.

To see that the directional derivative \eqref{eqn:second-directional-deriv-u} is also Lipschitz, fix any $v''\in\reals^d$, and consider the bound
\baligns
&\left|\grad_{v'}\grad_v u_h(x+v'') - \grad_{v'}\grad_v u_h(x)\right| \\
	\leq &\lim_{\epsilon' \to 0}\left| \frac{\grad_vu_h(x+v''+\epsilon' v')-\grad_vu_h(x+v'')}{\epsilon'} -  \frac{\grad_vu_h(x+\epsilon' v')-\grad_vu_h(x)}{\epsilon'}\right| \\ 
	\leq &\lim_{\epsilon' \to 0}\lim_{\epsilon \to 0}
		\bigg| \frac{u_h(x+v''+ \epsilon' v'+\eps v)-u_h(x+v''+\eps v)-(u_h(x+v''+\epsilon' v')-u_h(x+v''))}{\eps\epsilon'} \\
			&\quad\quad\quad-  \frac{u_h(x+\epsilon' v'+\eps v)-u_h(x+\eps v)-(u_h(x+\epsilon' v')-u_h(x))}{\eps\epsilon'}\bigg| \\ 
	\leq &\twonorm{v''}\left( M_1(h) \left(\frac{6L_3^2}{k^3}+\frac{L_4}{k^2}\right)
		+  M_2(h) \frac{3L_3}{k^2}
		+  M_3(h) \frac{2}{3k}\right),
\ealigns
where the final inequality follows from \lemref{third-order-u}.
Since each second-order directional derivative is Lipschitz continuous, we conclude that $u_h\in C^2(\reals^d)$ 
with Lipschitz continuous Hessian $\Hess u_h$.
Our Lipschitz gradient result \eqref{eqn:second-order-u-bound} and the Lipschitz relation \eqref{eqn:lipk} further furnish the uniform bound
$
M_2(u_h) %
	\leq  M_1(h) \frac{2L_3}{k^2}
		+  M_2(h) \frac{1}{k}. 
$
\paragraph{Solving the Stein equation}
Finally, we show that $u_h$ solves the Stein equation \eqref{eqn:stein-equation-generator}.
Introduce the notation $\transarg{t}{h}{x} \defeq \Earg{h(\process{t}{x})}$.
Since $(\trans{t})_{t\geq0}$ is strongly continuous 
on the Banach space $L$ of \lemref{langevin-properties} and $h\in L$,
the generator $\generator{}$, defined in \eqref{eqn:langevin-generator}, satisfies
\[
	h - \trans{t}{h} = \generator{\int_0^t \Esubarg{P}{h(\PVAR)} - \trans{s}{h}\, ds} \qtext{for all} t\geq 0
\]
by \cite[Prop. 1.5]{EthierKu86}.
The left-hand side limits in $L$ to $h - \Esubarg{P}{h(\PVAR)}$ as $t\to\infty$, as
\baligns
	|h(x) &- \Esubarg{P}{h(\PVAR)} - (h(x) - \transarg{t}{h}{x})|
		= \left|\int_{\reals^d} \Earg{h(\process{t}{y})} - \Earg{h(\process{t}{x})}\ p(y) dy\right| \\
		&\leq  M_1(h) \int_{\reals^d} \Earg{\twonorm{\process{t}{y} - \process{t}{x}}}\ p(y) dy
		\leq  M_1(h) \, \Esubarg{P}{\twonorm{Z-x}} e^{-kt/2}
\ealigns
for each $x\in\reals^d$ and $t\geq 0$. 
Here we have used the stationarity of $P$, the Lipschitz relation \eqref{eqn:lipk},
the first-order coupling inequality \eqref{eqn:contract1} of \lemref{synchronous}, and the integrability of $Z$~\cite[Lem. 1]{CuleSa10} in turn.
Meanwhile, the right-hand side limits to $\generator{u_h}$, since $\generator{}$ is closed \cite[Cor. 1.6]{EthierKu86}.
Therefore, $u_h$ solves the Stein equation \eqref{eqn:stein-equation-generator}.

\providecommand{\bysame}{\leavevmode\hbox to3em{\hrulefill}\thinspace}
\providecommand{\MR}{\relax\ifhmode\unskip\space\fi MR }
\providecommand{\MRhref}[2]{%
  \href{http://www.ams.org/mathscinet-getitem?mr=#1}{#2}
}
\providecommand{\href}[2]{#2}

\newcommand\acknowledgementbody{The authors thank Andreas Eberle for his suggestion to investigate triple couplings, the anonymous referee for his or her valuable feedback, and Andrew Duncan and Sebastian Vollmer for highlighting the need to identify an appropriate Banach space.
This material is based upon work supported by the National Science Foundation under Grant No.~1501767,
by the National Science Foundation Graduate Research Fellowship under Grant No.~DGE-114747,
and by the Frederick E.~Terman Fellowship.}

\vspace{\baselineskip}
\ACKNO{\acknowledgementbody}


\begin{thebibliography}{10}

\bibitem{Barbour88}
A.~D. Barbour, \emph{Stein's method and {P}oisson process convergence}, J.
  Appl. Probab. (1988), no.~Special Vol. 25A, 175--184, A celebration of
  applied probability. \MR{974580}

\bibitem{Barbour90}
\bysame, \emph{Stein's method for diffusion approximations}, Probab. Theory
  Related Fields \textbf{84} (1990), no.~3, 297--322. \MR{1035659}

\bibitem{Bentkus2003}
V.~Bentkus, \emph{On the dependence of the {B}erry-{E}sseen bound on
  dimension}, J. Statist. Plann. Inference \textbf{113} (2003), no.~2,
  385--402. \MR{1965117}

\bibitem{CattiauxGu14}
P.~Cattiaux and A.~Guillin, \emph{Semi log-concave {M}arkov diffusions},
  S\'eminaire de {P}robabilit\'es {XLVI}, Lecture Notes in Math., vol. 2123,
  Springer, Cham, 2014, pp.~231--292. \MR{3330820}

\bibitem{ChatterjeeMe08}
S.~Chatterjee and E.~Meckes, \emph{Multivariate normal approximation using
  exchangeable pairs}, ALEA Lat. Am. J. Probab. Math. Stat. \textbf{4} (2008),
  257--283. \MR{2453473}

\bibitem{ChatterjeeSh11}
S.~Chatterjee and Q.-M. Shao, \emph{Nonnormal approximation by {S}tein's method
  of exchangeable pairs with application to the {C}urie-{W}eiss model}, Ann.
  Appl. Probab. \textbf{21} (2011), no.~2, 464--483. \MR{2807964}

\bibitem{ChenGoSh11}
L.~H.~Y. Chen, L.~Goldstein, and Q.-M. Shao, \emph{Normal approximation by
  {S}tein's method}, Probability and its Applications (New York), Springer,
  Heidelberg, 2011. \MR{2732624}

\bibitem{CuleSa10}
M.~Cule and R.~Samworth, \emph{Theoretical properties of the log-concave
  maximum likelihood estimator of a multidimensional density}, Electron. J.
  Stat. \textbf{4} (2010), 254--270. \MR{2645484}

\bibitem{EthierKu86}
S.~N. Ethier and T.~G. Kurtz, \emph{Markov processes}, Wiley Series in
  Probability and Mathematical Statistics: Probability and Mathematical
  Statistics, John Wiley \& Sons, Inc., New York, 1986, Characterization and
  convergence. \MR{838085}

\bibitem{Friedman1975Stoch}
A.~Friedman, \emph{Stochastic differential equations and applications. {V}ol.                                                
  1}, Academic Press [Harcourt Brace Jovanovich, Publishers], New York-London,
  1975, Probability and Mathematical Statistics, Vol. 28. \MR{0494490}

\bibitem{GelmanCaStDuVeRu2014}
A.~Gelman, J.~B. Carlin, H.~S. Stern, D.~B. Dunson, A.~Vehtari, and D.~B.
  Rubin, \emph{Bayesian data analysis}, third ed., Texts in Statistical Science
  Series, CRC Press, Boca Raton, FL, 2014. \MR{3235677}

\bibitem{GorhamMa15}
J.~Gorham and L.~Mackey, \emph{Measuring sample quality with {S}tein's method},
  Advances in Neural Information Processing Systems 28 (C.~Cortes, N.~D.
  Lawrence, D.~D. Lee, M.~Sugiyama, and R.~Garnett, eds.), Curran Associates,
  Inc., 2015, pp.~226--234.

\bibitem{GorhamDuVoMa16}
J.~Gorham, A.~Duncan, S.~Vollmer, and L.~Mackey, \emph{Measuring sample quality
  with diffusions}, arXiv:1611.06972 (2016).

\bibitem{Gotze91}
F.~G{\"o}tze, \emph{On the rate of convergence in the multivariate {CLT}}, Ann.
  Probab. \textbf{19} (1991), no.~2, 724--739. \MR{1106283}

\bibitem{Khasminskii11}
R.~Khasminskii, \emph{Stochastic stability of differential equations}, second
  ed., Stochastic Modelling and Applied Probability, vol.~66, Springer,
  Heidelberg, 2012, With contributions by G. N. Milstein and M. B. Nevelson.
  \MR{2894052}

\bibitem{LedouxNoPe15}
M.~Ledoux, I.~Nourdin, and G.~Peccati, \emph{Stein's method, logarithmic
  {S}obolev and transport inequalities}, Geom. Funct. Anal. \textbf{25} (2015),
  no.~1, 256--306. \MR{3320893}

\bibitem{Meckes09}
E.~Meckes, \emph{On {S}tein's method for multivariate normal approximation},
  High dimensional probability {V}: the {L}uminy volume, Inst. Math. Stat.
  Collect., vol.~5, Inst. Math. Statist., Beachwood, OH, 2009, pp.~153--178.
  \MR{2797946}

\bibitem{NourdinPeRe10}
I.~Nourdin, G.~Peccati, and A.~R{\'e}veillac, \emph{Multivariate normal
  approximation using {S}tein's method and {M}alliavin calculus}, Ann. Inst.
  Henri Poincar\'e Probab. Stat. \textbf{46} (2010), no.~1, 45--58.
  \MR{2641769}

\bibitem{Pavliotis14}
G.~A. Pavliotis, \emph{Stochastic processes and applications}, Texts in Applied
  Mathematics, vol.~60, Springer, New York, 2014, Diffusion processes, the
  Fokker-Planck and Langevin equations. \MR{3288096}

\bibitem{ReinertRo09}
G.~Reinert and A.~R{\"o}llin, \emph{Multivariate normal approximation with
  {S}tein's method of exchangeable pairs under a general linearity condition},
  Ann. Probab. \textbf{37} (2009), no.~6, 2150--2173. \MR{2573554}

\bibitem{RobertsTw96}
G.~O. Roberts and R.~L. Tweedie, \emph{Exponential convergence of {L}angevin
  distributions and their discrete approximations}, Bernoulli \textbf{2}
  (1996), no.~4, 341--363. \MR{1440273}

\bibitem{Stein72}
C.~Stein, \emph{A bound for the error in the normal approximation to the
  distribution of a sum of dependent random variables}, Proceedings of the
  {S}ixth {B}erkeley {S}ymposium on {M}athematical {S}tatistics and
  {P}robability ({U}niv. {C}alifornia, {B}erkeley, {C}alif., 1970/1971), {V}ol.
  {II}: {P}robability theory, Univ. California Press, Berkeley, Calif., 1972,
  pp.~583--602. \MR{0402873}

\bibitem{Wishart1928}
J.~Wishart, \emph{The generalised product moment distribution in samples from a
  normal multivariate population}, Biometrika \textbf{20A} (1928), no.~1/2,
  32--52.

\end{thebibliography}
\end{document}